\documentclass[11pt]{amsart}
\usepackage{amsmath}
\usepackage{amssymb}
\usepackage[english]{babel}
\usepackage[latin1]{inputenc}
\usepackage[T1]{fontenc}
\usepackage{graphicx,epic,eepic}
\usepackage{color}
\usepackage{pdfpages}
\usepackage{enumerate}
\usepackage{amsthm}
\usepackage[all]{xy}
\usepackage{lastpage}
\usepackage{url}
\usepackage{caption}
\usepackage{subcaption}
\newtheorem{Theo}{Theorem}[section]
\newtheorem{Def}[Theo]{Definition}
\newtheorem{Ex}[Theo]{Example}
\newtheorem{Lem}[Theo]{Lemma}
\newtheorem{Prop}[Theo]{Proposition}
\newtheorem{Cor}[Theo]{Corollary}
\newtheorem{Con}[Theo]{Conjecture}
\newtheorem{Rem}[Theo]{Remark}
\usepackage{tikz}
\usepackage{tikz-qtree}

\newcommand{\xymat}[1]{\begin{align*}\xymatrix{ #1}\end{align*}}

\title[Homology of Algebras over the Spectral Lie Operad]{On the Odd Primary Homology of Free Algebras over the Spectral Lie Operad}
\author{Jens Jakob Kjaer}
\date{\today}
\address{Department of Mathematics,
255 Hurley, Notre Dame, IN 46556}

\begin{document}
\begin{abstract}
The derivatives of the identity functor on spaces in Goodwillie calculus form an operad in spectra. Antolin-Camarena computed the mod 2 homology of free algebras over this operad for 1-connected spectra. In this present paper we carry out similar computations for mod $p$ homology for odd primes $p$, also for non-connected spectra.
\end{abstract}

\thanks{The author thanks Mark Behrens for much help, as well as many personal conversations. Further the communication with the referee was most helpful in getting a better article written, and the author is very thankful for this. A frightening amount of typos and grammatical errors where removed with the help of Bridget Schreiner. The author was partially supported by NSF DMS 1209387}

\maketitle

\section*{Introduction}
It is known from \cite{CLM76} that if a spectrum is an algebra over an $E_\infty$ operad then it admits certain operations on its mod $p$ homology, the Dyer-Lashof operations. These have proven to be of great use in many computations. A different operad that is of interest to many topologists is the spectral Lie operad, $\partial_*$, and it has been shown in \cite{Beh11} that algebras over this operad admit certain Dyer-Lashof-like operations, in this paper called Lie power operations, on their mod 2 homology, as well as a Lie bracket. 

A strategy for computing unstable homotopy groups is through the Goodwillie tower (as defined in \cite{Goo03}) for the identity functor. This tower gives a spectral sequence inputting $\pi_*D_n(X)$ and converging to $\pi_*X$, for any 1-connected pointed space $X$, where $D_n(X)$ are some collection of infinite loop spaces, $D_n(X)=\Omega^\infty \mathbb{D}_n(X)$, whose homotopy we hope to be able to compute. The fact that the spaces $D_n(X)$ appear as $0$th spaces of spectra implies that we should be able to bring the full weight of stable computations to bear, hopefully allowing us to compute unstable homotopy groups from stable homotopy groups. An example of this program was carried out in  \cite{Beh12}. 

The spectra $\mathbb{D}_n(X)$ have the form $(\partial_n \wedge X^{\wedge n})_{h \Sigma_n}$, where $\partial_n$ is the $n$'th spectrum of the spectral Lie operad, since $\partial_n$ is in fact the $n$'th Goodwillie derivative of the identity functor. In \cite{AM99} the homology of $\mathbb{D}_n(S^l)$ was computed, and a basis was given in terms of unadmissable sequences of Dyer-Lashof operations. It was shown in \cite{Beh11} that these unadmissable sequences of Dyer-Lashof-like operations at the prime $2$ are in a very precise manner the same as the Lie power operations from above.

A different perspective on the symmetric sequence $\partial_*$ is given in \cite{Chi05}, where the operadic structure was constructed. There it is shown that $\partial_*$ is the Koszul dual of the commutative cooperad, which in fact gives its name of spectral Lie operad, as the Koszul dual of the algebraic commutative operad is the shifted algebraic Lie operad \cite{GiKa94}. These facts suggests that $\partial_*$ is an essential operad, and hence worthy of more study.  

In this paper we will study the odd primary homology of algebras over the spectral Lie operad. The definition of the Lie power operations and Lie bracket studied in this paper for the prime $p=2$ was first given by Behrens in \cite{Beh11}, and the computation of the homology of free algebras of $1$-connected spectra was carried out by Antolin-Camarena in \cite{Cam15} (rewritten and expanded as a paper \cite{Cam16}). Much of what follows here mirrors the strategies found in these references. A strengthening to computing  homology of free algebras of any spectrum is found in \cite{Bra17}.

We will discuss in section 1 the operadic structure of $\partial_*$, as well as give two different cell structures to the spectra $\partial_n$. The first is the arboreal cell structure, and comes from the operadic structure in \cite{Chi05}. The second is the simplicial cell structure and was leveraged in the homology computations of \cite{AM99}. We will also briefly identify the first differential in a certain spectral sequence that is of great use to us. 
In section 2 we will define the Lie power operations and Lie bracket, prove that the bracket in fact does deserve its name, as well as the fact that all power operations bracket trivially with everything. 
We will then move on, in section 3, to give a basis for the free algebra on an odd sphere using the computation in \cite{AM99}, and will then leverage this basis, as well as the study of the EHP sequence in \cite{Beh12}, to do the same for even spheres in positive dimensions, and then use \cite{Bra17} to extend it to negative dimensions as well. Thereafter we will be ready to state and prove our main result in section 4. The result is a basis for the homology of the free algebra over the spectral Lie operad on any spectrum. We will finish by stating a conjecture: that the relations the Lie power operations satisfy are the mixed Adem relations. This was shown in \cite{Beh12} for $p=2$, unfortunately the techniques used there to determine the relations are not known to generalize to odd primes.

\section{Preliminaries}
Throughout this article, let $p$ be a fixed odd prime. All homology is computed with $\mathbb{Z}/p$ coefficients. We will need a good symmetric monoidal category of spectra, so take spectra to mean the category of symmetric spectra as developed in \cite{HSS00}, and let $H\mathbb{F}_p$ be the Eilenberg-Maclane spectra for the finite field with $p$ elements, $\mathbb{F}_p$. 
\subsection{The Spectral Lie Operad}
We will let $\{\partial_*\}$ denote the symmetric sequence of spectra of the derivatives of the identity $Top_*\to Top_*$. Recall from \cite{Chi05} that $\{\partial_*\}$ has an operadic structure with structure maps in spectra:
\begin{align*}
\xi_{n,k_1,\ldots , k_n}:\partial_n\wedge \partial_{k_1} \wedge \ldots \wedge \partial_{k_n} \to \partial_{k_1+\ldots+k_n}.
\end{align*} 
By abuse of notation we will drop the indices on all of these maps throughout. As it is Koszul dual to the commutative cooperad in spectra, we will call $\{\partial_*\}$ the spectral Lie operad.

Our main target of computation is the homology of the following
\begin{Def}
Let $X$ be any spectrum. Then the free algebra over the spectral Lie operad generated by $X$, $\mathbb{P}(X)$, is given by $\bigvee_n\mathbb{D}_n(X)$, where $\mathbb{D}_n(X):=(\partial_n\wedge X^{\wedge n})_{h\Sigma_n}$ where $\Sigma_n$ permutes the copies of $X^{\wedge n}$.
\end{Def}
Clearly the free algebra over the spectral Lie operad is in fact an algebra over $\partial_*$.

\subsubsection*{Cell Structures and Trees}
As the goal is to do computations in homology, a good understanding of the cell structure of $\partial_n$ is needed. We will employ two different such. The first one, called the arboreal cell structure, comes from Ching's description of the topological operadic bar construction, and is therefore well-behaved with respect to the operadic structure maps. The second, called the simplicial cell structure, comes from the cosimplicial filtration of the operadic bar construction, this has proved valuable in the homology computations carried out by Arone and Mahowald \cite{AM99}, but is not well behaved with the operadic structure. In either case, the cells will be labelled by certain trees.   

When we write a general tree with labels in some finite set $A$, we will always mean a rooted tree with the valence of the root and leaves being 1, and a fixed bijection from $A$ to the set of leaves. We will often suppress the set $A$ from the notation. An internal vertex is any vertex that is not a leaf or the root. We will call it a tree if the valence of any internal vertex is greater than 2. Given a general tree $T$, we will use $E(T)$ to denote the set of edges, and $V(T)$ to denote the set of vertices. Given $u,v\in V(T)$ we say that $v$ is a descendant of $u$ if there is an edge from $u$ to $v$ and $u$ is closer to the root than $v$. 
\begin{Def}
A metric tree is a tree, $T$, together with a map \linebreak$m:E(T)\to [0,1]$ such that if $e_1,\ldots, e_n$ is any path from the root to a leaf then $\Sigma_{i=1}^nm(e_i)=1$.
\end{Def}
Recall that \cite[Def. 4.1]{Chi05} defines pointed topological spaces $\partial^n$, such that the Spanier-Whitehead dual of $\partial^n$ is $\partial_n$. A point of $\partial^n$ is either given by a metric tree with labels $\{1,\ldots , n\}$, or the base point, where we identify a tree with an edge of metric 0 with the tree where we have collapsed this edge, if it is an internal edge or with the basepoint if it is a root or leaf edge. Ching defined the operadic structure in the language of metric trees. Ching's arboreal construction of $\partial^n$ leads to a cell decomposition of $\partial_n$, where each cell is labelled by an isomorphism class of a labelled trees, with labels $\{1,\ldots, n\}$. The dimension of the cell represented by such a tree is given by $-k$ where $k$ is the number of internal vertices. We will call this cell decomposition the arboreal cell decomposition. The operadic structure induces maps
\begin{align*}
\xi_*: C^{CW}_*(\partial_n)\otimes C^{CW}_*(\partial_{k_1})\otimes \ldots \otimes C^{CW}_*(\partial_{k_n})\to C^{CW}_*(\partial_{k_1+\ldots +k_n}) 
\end{align*}
that take $T\otimes T^{(1)}\otimes \ldots \otimes T^{(n)}$ to the cell labelled by the tree obtained by identifying the root edge of $T^{(i)}$ with the edge attached to the leaf of $T$ labelled $i$ (See the example in Figure \ref{fig:op}).

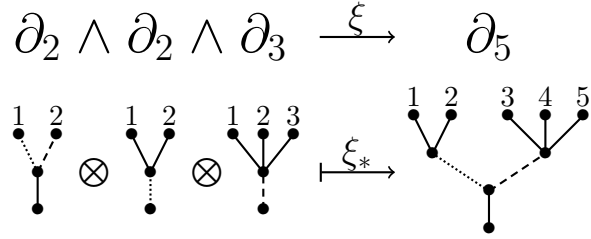
\begin{figure}[h] 
\begin{center}   
   \begin{tikzpicture}[xscale=1,yscale=1]
\draw (0, 2.3) node {\huge $\partial_2$};
\draw (1.5, 2.3) node {\huge $\partial_2$};
\draw (3, 2.3) node {\huge $\partial_3$};
\draw (6, 2.3) node {\huge $\partial_5$};
\draw [->,thick] (3.75, 2.3) -- (4.75 , 2.3);
\draw (4.25, 2.55) node {\Large $\xi$};
\draw (0.75, 2.3) node {\huge $\wedge$};
\draw (2.25, 2.3) node {\huge $\wedge$};
\draw (0, 0) node  {$\bullet$};
\draw (0, 0.5) node  {$\bullet$};
\draw (-0.25, 1) node  {$\bullet$};
\draw (0.25, 1) node  {$\bullet$};
\draw (-0.25, 1.25) node  {$1$};
\draw (0.25, 1.25) node  {$2$};
\draw [-,thick] (0, 0) -- (0, 0.5);
\draw [-,thick,densely dashed] (0, 0.5) -- (0.25 , 1);
\draw [-,thick,densely dotted] (0, 0.5) -- (-0.25 , 1);
\draw (0.75, 0.5) node {$\bigotimes$};
\draw (1.5, 0) node  {$\bullet$};
\draw (1.5, 0.5) node  {$\bullet$};
\draw (1.25, 1) node  {$\bullet$};
\draw (1.75, 1) node  {$\bullet$};
\draw (1.25, 1.25) node  {$1$};
\draw (1.75, 1.25) node  {$2$};
\draw [-,thick,densely dotted] (1.5, 0) -- (1.5, 0.5);
\draw [-,thick] (1.5, 0.5) -- (1.75 , 1);
\draw [-,thick] (1.5, 0.5) -- (1.25 , 1);
\draw (2.25, 0.5) node {$\bigotimes$};
\draw (3, 0) node  {$\bullet$};
\draw (3, 0.5) node  {$\bullet$};
\draw (2.6, 1) node  {$\bullet$};
\draw (3.4, 1) node  {$\bullet$};
\draw (3, 1) node  {$\bullet$};
\draw (2.6, 1.25) node  {$1$};
\draw (3.4, 1.25) node  {$3$};
\draw (3, 1.25) node  {$2$};
\draw [-,thick,densely dashed] (3, 0) -- (3, 0.5);
\draw [-,thick] (3, 0.5) -- (2.6 , 1);
\draw [-,thick] (3, 0.5) -- (3 , 1);
\draw [-,thick] (3, 0.5) -- (3.4 , 1);
\draw [|->,thick] (3.75, 0.5) -- (4.75 , 0.5);
\draw (4.25, 0.75) node {\Large $\xi_*$};
\draw (6, -0.25) node {$\bullet$};
\draw (6, 0.25) node {$\bullet$};
\draw (5.25, 0.75) node {$\bullet$};
\draw (6.75, 0.75) node {$\bullet$};
\draw (5, 1.25) node {$\bullet$};
\draw (5.5, 1.25) node {$\bullet$}; 
\draw (5, 1.5) node {$1$};
\draw (5.5, 1.5) node {$2$};
\draw (6.75, 1.25) node {$\bullet$};
\draw (6.25, 1.25) node {$\bullet$};
\draw (7.25, 1.25) node {$\bullet$};
\draw (6.75, 1.5) node {$4$};
\draw (6.25, 1.5) node {$3$};
\draw (7.25, 1.5) node {$5$};
\draw [-,thick]  (6, -0.25) --  (6, 0.25);
\draw [-,thick,densely dotted]  (6, 0.25) -- (5.25, 0.75);
\draw [-,thick,densely dashed]  (6, 0.25) -- (6.75, 0.75);
\draw [-,thick]  (5.25, 0.75) -- (5, 1.25);
\draw [-,thick]  (5.25, 0.75) -- (5.5, 1.25);
\draw [-,thick]  (6.75, 0.75) -- (6.75, 1.25);
\draw [-,thick]  (6.75, 0.75) -- (6.25, 1.25);
\draw [-,thick]  (6.75, 0.75) -- (7.25, 1.25);
\end{tikzpicture} 
  \end{center}
\caption{\label{fig:op} An example of the operadic structure on arboreal cells. The dotted and dashed edges indicate the gluing.}
\end{figure}

\begin{Def}
A levelled tree is a general tree, $T$, together with a function \linebreak$l:V(T)\to [0,1]$ such that the root goes to 0, the leaves to 1, and if $v$ is a descendant of $u$ then $l(u)< l(v)$. The number of levels of a levelled tree is $|l(V(T))|-2$, where $|-|$ denotes cardinality.  

We say that two levelled trees $T$ and $T'$ are isomorphic if there is an isomorphism of the trees, and a strictly increasing function $\phi:[0,1]\to [0,1]$ making the obvious relations hold.
\end{Def}
There is a different cell structure on $\partial_n$ where each cell is given by an isomorphism class of levelled trees with labels $\{1,\ldots, n\}$, as given in \cite{Beh12}. The cell represented by the class of a levelled tree, $T$, is in dimension $-k$ where $k$ is the number of levels, and the gluing data comes from collapsing levels. We will call this cell-decomposition the simplicial cell decomposition.

Note that in both cell decompositions, $\partial_n$ has exactly one $(-1)$-cell, for $n>1$. The cell is given by the tree having exactly one internal vertex, we call this tree $T_n$.

\subsubsection*{G-Cells}
Recall that a naive $G$-CW complex is a naive $G$-spectrum $X$ whose underlying spectrum is a CW-spectrum, such that if $X^{(k)}$ is the $k$-th skeleta, then $X^{(k)}$ is a naive $G$-spectrum, and $X^{(k-1)}\hookrightarrow X^{(k)}$ is a $G$-map with $G$-cofiber $\bigvee_i (G/H_i)_+ \wedge S^{k}$, for some collection of subgroups $\{H_i\}$ of $G$. For each $i$ we say that $X$ has a $(k,H_i)$-cell. 

Goodwillie showed in \cite{Goo03} that $\partial_n$ is a naive $\Sigma_n$ spectrum. Clearly the $\Sigma_n$ action of trees with leaves labelled by $\{1,\ldots , n\}$ induces an action on the cells of $\partial_n$, making it into a naive $\Sigma_n$-CW complex. If $G$ is any subgroup of $\Sigma_n$, then clearly $\partial_n$ is a $G$-CW spectrum. If $\Sigma_T\subset \Sigma_n$ is the symmetry group of the tree $T$ labelled by $\{1,\ldots , n\}$, let $\Sigma_T'=G\cap \Sigma_T$. Then $\partial_n$ has a $(k,\Sigma_T')$-cell, where $k$ is the number of internal vertices of $T$. Note that if $T'$ is any tree such that there is $g\in G$ such that $g\cdot T'=T$, then $T'$ represents the same $G$-cell as $T$. This discussion works equally well with arboreal or simplicial cells.

\subsubsection*{A Small Forest of Examples}
\begin{Def}
Let $T^{j,i}$ for $1\leq j \leq i$ denote the labelled tree, with labels $\{1,\ldots, i\}$, depicted below
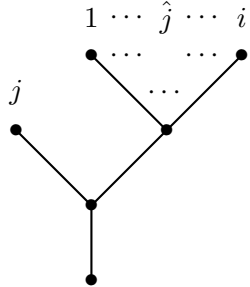
\begin{align*}
\begin{tikzpicture}[xscale=1,yscale=1]
\draw (0, 0) node [fill=white] {$\bullet$};
\draw (0, 1) node [fill=white] {$\bullet$};
\draw (1 , 2) node [fill=white] {$\bullet$};
\draw (-1 , 2) node [fill=white] {$\bullet$};
\draw (-1 , 2.5) node [fill=white] {$j$};
\draw (0.0 , 3.5) node [fill=white] {$1$};
\draw (0.5 , 3.5) node [fill=white] {$\cdots$};
\draw (1 , 3.5) node [fill=white] {$\hat{j}$};
\draw (1.5 , 3.5) node [fill=white] {$\cdots$};
\draw (2 , 3.5) node [fill=white] {$i$};
\draw (0.0 , 3) node [fill=white] {$\bullet$};
\draw (0.5 , 3) node [fill=white] {$\cdots$};
\draw (1.5 , 3) node [fill=white] {$\cdots$};
\draw (2 , 3) node [fill=white] {$\bullet$};
\draw [-,thick] (0, 0) -- (0, 1);
\draw [-,thick] (0, 1) -- (-1, 2);
\draw [-,thick] (0, 1) -- (1, 2);
\draw [-,thick] (1, 2) -- (0, 3);
\draw [-,thick] (1, 2) -- (2, 3);
\draw (1 , 2.5) node  {$\cdots$};
\end{tikzpicture}
\end{align*}  
where $\hat{j}$ denotes that the label $j$ is omitted.

In the arboreal structure this represents a $-2$ cell.
\end{Def} 
\begin{Def}
Let $T_{n,k}$ be a levelled tree with labels $\{1,\ldots, n^k\}$, where each internal vertex has $n$-descendants, and $k$-levels.

In the simplicial cell structure this represents a $(-k)$-cell.
\end{Def}
Note that $T_{n,1}=T_n$.
 \begin{Ex}
 The tree $T_{3,2}$ is:
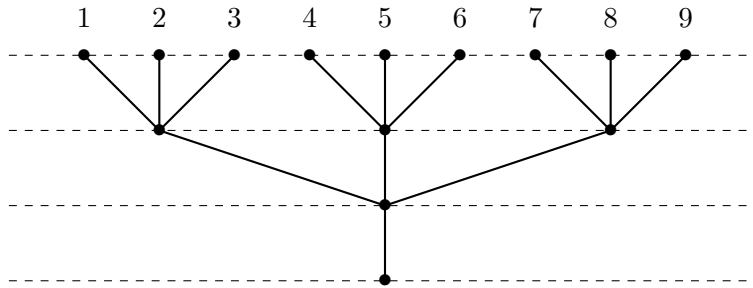
\begin{align*}
\begin{tikzpicture}[xscale=1,yscale=1]
\draw (0, 0) node  {$\bullet$};
\draw (0, 1) node {$\bullet$};
\draw (0 , 2) node  {$\bullet$};
\draw (-3 , 2) node {$\bullet$};
\draw (3 , 2) node  {$\bullet$};
\draw (-4 , 3) node  {$\bullet$};
\draw (-3 , 3) node  {$\bullet$};
\draw (-2 , 3) node  {$\bullet$};
\draw (-1 , 3) node  {$\bullet$};
\draw (0 , 3) node  {$\bullet$};
\draw (1 , 3) node  {$\bullet$};
\draw (4 , 3) node  {$\bullet$};
\draw (3 , 3) node  {$\bullet$};
\draw (2 , 3) node  {$\bullet$};
\draw (-4 , 3.5) node  {$1$};
\draw (-3 , 3.5) node  {$2$};
\draw (-2 , 3.5) node  {$3$};
\draw (-1 , 3.5) node  {$4$};
\draw (0 , 3.5) node  {$5$};
\draw (1 , 3.5) node  {$6$};
\draw (4 , 3.5) node  {$9$};
\draw (3 , 3.5) node  {$8$};
\draw (2 , 3.5) node  {$7$};
\draw [-,thick] (0, 0) -- (0, 1);
\draw [-,thick] (0, 1) -- (-3, 2);
\draw [-,thick] (0, 1) -- (3, 2);
\draw [-,thick] (0, 1) -- (0, 2);
\draw [-,thick] (-3, 2) -- (-4, 3);
\draw [-,thick] (-3, 2) -- (-3, 3);
\draw [-,thick] (-3, 2) -- (-2, 3);
\draw [-,thick] (0, 2) -- (-1, 3);
\draw [-,thick] (0, 2) -- (0, 3);
\draw [-,thick] (0, 2) -- (1, 3);
\draw [-,thick] (3, 2) -- (2, 3);
\draw [-,thick] (3, 2) -- (3, 3);
\draw [-,thick] (3, 2) -- (4, 3);
\draw [-,dashed] (-5, 2) -- (5, 2);
\draw [-,dashed] (-5, 1) -- (5, 1);
\draw [-,dashed] (-5, 0) -- (5, 0);
\draw [-,dashed] (-5, 3) -- (5, 3);
\end{tikzpicture}
\end{align*}   
where the dashed horizontal lines indicate levels.
 \end{Ex}
 
 \subsection{Computational Methods}
 In this paper we will often  need to calculate the homology groups $H_*(\partial_n\wedge_{hG} X)$ for some $G\subset \Sigma_n$ where $X$ is a $G$-spectrum. We will use the following spectral sequence.
\begin{Lem}\label{SS}
There is a spectral sequence coming from the arboreal cell decompostion of $\partial_n$ with
\begin{align*}
E^1_{k,*}:=\bigoplus_{T}H_*(\Sigma^{-k}X_{h\Sigma_T'})\Rightarrow H_*(\partial_n\wedge_{hG} X)
\end{align*}
where the direct sum is over $G$ equivalence classes of trees $T$ with $k$ internal vertices, and $\Sigma_T':=G\cap \Sigma_T$, and differentials $d_r:E^r_{k,n}\to E^r_{k-r,*-1}$.
\end{Lem}  
The spectral sequence arising from the $G$-CW filtration of $\partial_n$ induces a filtration of $\partial_n\wedge_{hG} X$, and we apply homology to this filtration to obtain the spectral sequence.

   The following lemma allows us to identify the $d_1$-differentials: 
 \begin{Lem}\label{d_1=Tr}
Let $G$ be a subgroup of $\Sigma_n$, and $X$ be a $G$-spectrum. In the spectral sequence computing $H_*(\partial_{n}\wedge_{hG} X)$, from the $G$-equivariant arboreal cell decomposition of $\partial_n$, we have that $d_1$ is a sum of transfers and trivial maps.
\end{Lem}
\begin{proof}
Recall that $d_1$ is given by
\begin{align*}
 \bigoplus_{T} H_*(S^{-s}\wedge X_{h\Sigma_{T}'}) \stackrel{\delta}{\to}  H_{*-1}(\partial_n^{(-s-1)} \wedge_{hG} X) \to \bigoplus_{T'} H_{*-1}(S^{-s-1}\wedge X_{h\Sigma_{T'}'})
\end{align*}
where the first direct sum is over $(-s)$-$G$-cells of $\partial_n$, labelled by orbits of trees $T$ under the $G$ action, and the second is over $(-s-1)$ cells, labelled by orbits of trees $T'$ under the $G$ action, and $\partial_n^{(-s-1)}$ is the $(-s-1)$-skeleton of $\partial_n$. Then $\delta$ comes from the cofiber sequence 
\xymat{\partial_n^{(-s-1)}\ar[r] & \partial_n^{-s} \ar[r] & \bigvee_T S^{-s} \ar[r]^\delta & \Sigma \partial_n^{(-s-1)},}
 and $\Sigma_T',\Sigma_{T'}'\subset G$ are as above.
Let $T$ be a tree obtained from $T'$ be collapsing an edge, and assume that $T'$ has $s+1$ internal vertices (and hence $T$ has $s$). Note that this further implies that $\Sigma_{T'}\subset \Sigma_T$. Therefore our goal is now showing that the restriction
\begin{align*}
  H_*(S^{-s}\wedge X_{h\Sigma_{T}'})  \stackrel{\bar{d_1}}{\to}  H_{*-1}(S^{-s-1}\wedge X_{h\Sigma_{T'}'})
\end{align*}
is in fact the transfer map. If we run the spectral sequence computing $H_*(\partial_n\wedge X)$ from the cells of $\partial_n$ then the analogous map $H_*(S^{-s}\wedge X) \to  H_{*-1}(S^{-s-1}\wedge X)$, coming from the attaching map between $T$ and $T'$, would be the identity. Note for all $\sigma\in G$, the cell represented by $\sigma \cdot T$ is obtained from $\sigma \cdot T'$ by collapsing an edge. Note that we can think of $\bar{d_1}$ as induced by applying $(-)_{hG}$ to
\begin{align*}
\bigvee_{\sigma\in G/\Sigma_T'} X \to \bigvee_{\sigma \in G/\Sigma_{T'}'}X
\end{align*}
which is equivalent to applying $(-)_{h\Sigma_T'}$ to
\begin{align*}
 X \to \bigvee_{\sigma \in \Sigma_T'/\Sigma_{T'}'}X,
\end{align*}
which is exactly the definition of the transfer.
\end{proof}

\section{Homology Operations}
In this section we will define and prove certain relations for the Lie power operations, as well as a Lie bracket. Let $L$ be an algebra over the operad $\partial_*$, with structure maps $\xi_n: \partial_n\wedge L^{\wedge n} \to L$. We will again drop all indices, and thus use $\xi$ for structure maps for both the operad itself and its algebras.
\subsection{Power Operations}
We wish to define Lie power operations 
\begin{align*}
\overline{\beta^\epsilon Q^i}: H_*(L)\to H_{*+2(p-1)i-\epsilon-1}(L).
\end{align*}

Recall from \cite[Thm I.1.1]{CLM76} that for any spectrum $X$, $i\in \mathbb{N}_0$, and $\epsilon$ equal to $0$ or $1$ we have maps $q_{i,\epsilon}: H_*(X)\to H_{*+2(p-1)i-\epsilon}(X^{\wedge p}_{h\Sigma_p})$, taking $x\mapsto \beta^\epsilon Q^i(x)$. 
This map comes from a study of $C_*^{CW}(E\Sigma_p)$, which has certain elements $e_k$ of degree $k$, such that when $k=(p-1)j-\epsilon$ and $x\in C_*(X)$ is a cycle then so is $e_k\otimes x^{\otimes p}\in C_*^{CW}(E\Sigma_p)\otimes_{\Sigma_p} C_*(X)^{\otimes p}$, when $j$ has the same parity as $|x|$. This cycle represents the class $\beta^\epsilon Q^{\frac{j+|x|}{2}}(x)\in H_{*}(X^{\wedge p}_{h\Sigma_p})$. Ideally, one would wish to recreate this construction by picking out explicit cycles in the chain complex for $\partial_p\wedge_{h\Sigma_p}L^{\wedge p}$, unfortunately this turns out not to be feasible, and we will instead attack the problem somewhat indirectly.

Given an element $x\in H_*(L)$ we wish to define an element $T_p\otimes \beta^\epsilon Q^{i}(x)$ in the homology group $H_{*+2(p-1)i-\epsilon-1}(\partial_p\wedge_{h\Sigma_p}L^{\wedge p})$, given by the cell represented by the tree $T_p$ in either cell description of $\partial_p$.
\begin{Lem} \label{Lem:Syl}
Let $\iota\in H_*(S^j)$ be the generator. The element $\sigma^{-1}\beta^\epsilon Q^i(\iota)$ in $ H_*(\Sigma^{-1}(S^j)^{\wedge p}_{h\Sigma_{T_p}})$ survives the spectral sequence from Lemma \ref{SS} used to compute $H_*(\partial_p\wedge_{h\Sigma_p}(S^j)^{\wedge p})$.
\end{Lem}
\begin{proof}
Assume that $T$ 
is a tree representing a $-k\neq -1$ cell of $\partial_p$. Then $T\neq T_p$, and clearly $p\nmid |\Sigma_T|$. We can compute 
$H_*(\Sigma^{-k}(S^j)^{\wedge p}_{h\Sigma_T})$ by the usual homotopy orbits spectral sequence with $E^2$ page $E^2_{s,t}=H_s(\Sigma_T;H_t((S^j)^{\wedge p}))$. 

Now if $j$ is odd then $H_*((S^j)^{\wedge p})$ is the sign representation and hence $H_*(\Sigma^{-k}(S^j)^{\wedge p}_{h\Sigma_T})$ is trivial
by \cite[Cor. 6.5.9]{Wei95}. Therefore the spectral sequence from  Lemma \ref{SS} collapses and the result holds.

If $j$ is even, then $H_*((S^j)^{\wedge p})$ is the trivial representation and hence we see that $H_*(\Sigma^{-k}(S^j)^{\wedge p}_{h\Sigma_T})$ is concentrated in degree $jp-k$. 
If $T$ represents a $(-2)$-cell then by Lemma \ref{d_1=Tr} we know that 
\begin{align*}
d_1:H_*\Sigma^{-1}(S^j)^{\wedge p}_{h\Sigma_{T_p}}\to H_{*-1}\Sigma^{-2}(S^j)^{\wedge p}_{h\Sigma_{T}}
\end{align*} is induced by the transfer coming from $\Sigma_T\subset \Sigma_{T_p}=\Sigma_p$. Since $p$ divides $[\Sigma_p:\Sigma_T]$ we know that the transfer is trivial, since the inclusion, 
\begin{align*}
H_*(\Sigma_T)\to H_*(\Sigma_p),
\end{align*}
 composed with the transfer induces the map multiplication by $[\Sigma_T:\Sigma_p]$, and we know the inclusion is non trivial. 
If $T$ represents a $(-k-1)$-cell, then  by studying the spectral sequence we see by induction that 
\begin{align*}
d_k:H_*\Sigma^{-1}(S^j)^{\wedge p}_{h\Sigma_{T_p}}\to E^k_{k+1,*-1}.
\end{align*}
 Here the target is concentrated in degree $jp-k-2$, since $H_{*-1}\Sigma^{-k-1}(S^j)^{\wedge p}_{h\Sigma_{T}}$ is, but clearly the source is concentrated in higher degrees, and hence the map is trivial, and therefore we are done.
\end{proof}

Given $x\in H_i(L)$, 
we can represent it by a map $S^i\to H\mathbb{F}_p\wedge L$, which corresponds to a map out of the free $H\mathbb{F}_p$-module $\Sigma^iH\mathbb{F}_p$, which by abuse of notation we are also going to call $x$, so $x:\Sigma^iH\mathbb{F}_p\to H\mathbb{F}_p\wedge L$. This gives us a map 
\begin{align*}
x^{\otimes p}:(\Sigma^iH\mathbb{F}_p)^{\otimes_{H\mathbb{F}_p}p}\to (H\mathbb{F}_p\wedge L)^{\otimes_{H\mathbb{F}_p}p}
\end{align*}
by smashing with $\partial_p$, and taking homotopy orbits we get a map 
\begin{align*}
\tilde{x}:H_*(\partial_p\wedge_{h\Sigma_p} S^{pi})\to H_*(\partial_p\wedge_{h\Sigma_p} L^{\wedge p}).
\end{align*}
 We can now define an element $T_p\otimes \beta^\epsilon Q^i (x)$ in $ H_*(\partial_p\wedge_{h\Sigma_p} L^{\wedge p})$ by $\tilde{x}(\sigma^{-1}\beta^\epsilon Q^i(\iota))$.
\begin{Def}
For $x\in H_*(L)$, and $\xi: \partial_p\wedge_{h\Sigma_p} L^{\wedge p}\to L$ define: 
\begin{align*}
\overline{\beta^\epsilon Q^i}(x):=\xi_*(T_p\otimes \beta^\epsilon Q^{i}(x))\in H_{*+2(p-1)i-\epsilon-1}(L).
\end{align*}
\end{Def}

It is easy to see that if $ 2n-1\leq |x|$, then both $\overline{\beta Q^i}(x)$ and $\overline{Q^j}(x)$ are trivial for $i,j<n$. Since the same is true of $\beta Q^i \iota$ and $Q^j \iota$ in $H_*(S^{k p}_{h\Sigma_p})$, where $k=|x|$. Furthermore the quotient on to the top cell, $T_p$, $\partial_p\to S^{-1}$ induces a map $H_*(\partial_p\wedge_{h\Sigma_p} X^{\wedge p})\to H_*(\Sigma^{-1} X^{\wedge p}_{h\Sigma_p})$, which maps $T_p\otimes \beta^\epsilon Q^i (x)$ to $\sigma^{-1}\beta^\epsilon Q^i x$, where $\sigma^{-1}$ refers to the desuspension.

\subsection{The Bracket}
We can also define a Lie bracket on the algebra by the map $\xi: \partial_2\wedge L \wedge L \to L$
\begin{Def}
For $x\in H_i(L)$, and $y\in H_j(L)$, define $[x,y]\in H_{i+j-1}(L)$ by $[x,y]:=\xi_*(T_2 \otimes x \otimes y)$. 
\end{Def} 
\begin{Rem}
Note this construction works equally well on stable homotopy groups but studying these, and their possible connection with the Whitehead bracket, is beyond the scope of this paper. 
\end{Rem}
\begin{Prop}\label{LieBrac}
The bracket satisfies the following relations for $x,y,z\in H_*(L)$:
\begin{itemize}
\item $[x,y]=(-1)^{|x||y|}[y,x]$ (Graded Commutativity)
\item $(-1)^{|x||z|}[x,[y,z]] + (-1)^{|y||z|}[y,[z,x]]+(-1)^{|z||y|}[z,[x,y]]=0$ (The Graded Jacobi Identity)
\end{itemize}
\end{Prop}
\begin{Rem}
Note that our Lie bracket does not satisfy the usual conventions for either the graded Lie bracket (due to dimension) or the shifted graded Lie bracket (due to sign conventions, see for example \cite{KM95} for a different sign convention). We will later see in Corollary \ref{[x,[x,x]]} that for $p=3$ the usual axiom $[x,[x,x]]=0$ still holds.
\end{Rem}
\begin{proof}[Proof of Prop. \ref{LieBrac}]
We clearly see that $[x,y]=(-1)^{|x||y|}[y,x]$ since $\partial_2\simeq S^{-1}$ with the trivial $\Sigma_2$ action. 
Since $L$ is an algebra over  $\partial_*$, the following diagram commutes:
\xymat{\partial_2 \wedge L \wedge (\partial_2 \wedge L \wedge L) \ar[r]^{1\wedge \xi} \ar[d]^{\simeq} & \partial_2 \wedge L \wedge L \ar[r]^\xi & L \\
\partial_2 \wedge (\partial_1 \wedge L) \wedge (\partial_2 \wedge L \wedge L) \ar[d]^\sigma & &\\
\partial_2 \wedge \partial_1 \wedge \partial_2 \wedge L \wedge L \wedge L \ar[d] & & \\
\partial_3 \wedge L \wedge L \wedge L \ar@/_2pc/[rruuu]_\xi 
}
Pick $x,y,z\in H_*(L)$ and apply homology to the diagram. If we start with $T_2\otimes x \otimes (T_2 \otimes y \otimes z)$ in the upper left corner, we get $[x,[y,z]]$ in the upper right corner. In the lower left corner we get by the operadic structure $T^{1,3}\otimes x \otimes y \otimes z$. Using the permutation action we see that $(-1)^{|x||y|+|x||z|}T^{2,3}\otimes x\otimes y \otimes z$ maps to $[y,[z,x]]$ since the diagram above is $\Sigma_3$-equivariant, and in the same manner $(-1)^{|z||y|+|x||z|}T^{2,3}\otimes x\otimes y \otimes z$. Now, by abuse of notation, let the trees $T_3$ and $T^{i,3}$ $i=1,2,3$ be the cells of $\partial_3$ in the arboreal cell decomposition. Then the chain differential is $d(T_3)=T^{1,3}+T^{2,3}+T^{3,3}$, so if we let $x,y,z$ also denote cycles in the singular chain complex for $L$ then the following boundary $d((-1)^{|x||z|}T_3\otimes x \otimes y \otimes z)$ in the chain complex computing $H_*(\partial_3\wedge L \wedge L \wedge L)$ enforces the relation:
\begin{align*}
(-1)^{|x||z|}[x,[y,z]] + (-1)^{|y||z|}[y,[z,x]]+(-1)^{|z||y|}[z,[x,y]]=0. 
\end{align*}
\end{proof}
We have further the following interaction of the bracket and the Lie power operations:
\begin{Prop}
For all $x,y\in H_*L$, $k\in \mathbb{N}_0$, and $\epsilon=0,1$, we have $[x,\overline{\beta^\epsilon Q^k}y]=0$. 
\end{Prop}
\begin{proof}
This is the odd primary version of Lemma 6.5 in \cite{Cam15}, and we will start out similar. 

Let $x: \Sigma^iH\mathbb{F}_p\to H\mathbb{F}_p\wedge L$ and $y: \Sigma^jH\mathbb{F}_p\to H\mathbb{F}_p\wedge L$ represent $x$ and $y$. 
By studying the arboreal cell structure we see that:
\begin{align*}
\xi_*: C^{CW}_*(\partial_2\wedge \partial_1 \wedge \partial_p) & \to C^{CW}_*(\partial_{1+p}) \\
T_2\otimes T_1\otimes T_p & \mapsto T^{1,1+p}  
\end{align*}
and hence we get the following commutative $1\times \Sigma_p$ equivariant diagram of 
\xymat{C^{CW}_*(\partial_{2})\otimes C_*(S^i)  \otimes C^{CW}_*(\partial_{p}) \otimes C_*(S^j)^{\otimes p}  \ar[d] \ar[r]& C^{CW}_*(\partial_{2})\otimes C_*(L)  \otimes C^{CW}_*(\partial_{p}) \otimes C_*(L)^{\otimes p} \ar[d]   \\ C^{CW}_*(\partial_{1+p})\otimes C_*(S^i) \otimes C_*(S^j)^{\otimes p}\ar[r] & C^{CW}_*(\partial_{1+p})\otimes C_*(L) \otimes C_*(L)^{\otimes p} \ar[d] \\
& C_*(L)}
If we start with $T_2\otimes x\otimes T_p\otimes y^{\otimes p}$ in the upper left corner, then it is mapped to a class that represents $[x,\overline{\beta^\epsilon Q^k}y]$ in $H_*(L)$. By taking Borel homology, with respect to $1\times \Sigma_p$, everywhere we get that it is enough to show that the image of $T^{1,1+p}\otimes x\otimes y$ in $H_*(\partial_{1+p}\wedge S^i \wedge (S^j)^{\wedge p})_{h1\times\Sigma_p}$ is trivial. Now we are going to diverge from the proof of Lemma 6.5 in \cite{Cam15}, as we are not going to compute the homology groups in their entirety.
We will need the spectral sequence coming from the arboreal $(1\times \Sigma_p)$-cell decomposition of $\partial_{1+p}$ (see Figure \ref{fig:d4} for the example $p=3$, with $T^{1,4}$ in bold), and its $d_1$-differentials.

\begin{figure}[h]
\begin{center}
    \begin{tabular}{ c | c }
   Dimension & Cells \\ \hline
   $-1$ &   \begin{subfigure}[h]{5cm}
   \begin{tikzpicture}[xscale=1,yscale=1]
\draw (0, 0) node  {\small $\bullet$};
\draw (0, 0.3) node  {\small $\bullet$};
\draw (-0.3 , 0.6) node  {\small $\bullet$};
\draw (0.3 , 0.6) node  {\small $\bullet$};
\draw (0.1 , 0.6) node  {\small $\bullet$};
\draw (-0.1 , 0.6) node  {\small $\bullet$};
\draw (-0.3, 0.8) node {\tiny $1$};
\draw (0, -0.25) node {\scriptsize $\Sigma_3$};
\draw [-,semithick] (0,0) -- (0,0.3);
\draw [-,semithick] (0,0.3) -- (-0.3,0.6);
\draw [-,semithick] (0,0.3) -- (0.3,0.6);
\draw [-,semithick] (0,0.3) -- (-0.1,0.6);
\draw [-,semithick] (0,0.3) -- (0.1,0.6);
\end{tikzpicture} 
\end{subfigure} \\ \hline
$-2$ & \begin{subfigure}[h]{5cm}
   \begin{tikzpicture}[xscale=1,yscale=1]
\draw (0, 0) node  {$\bullet$}; 
\draw (0, 0.3) node  {$\bullet$};
\draw (-0.3 , 0.6) node  {$\bullet$};
\draw (0.3 , 0.6) node  {$\bullet$};
\draw (0.3, 0.9 ) node {$\bullet$};
\draw (0.1, 0.9 ) node {$\bullet$};
\draw (0.5, 0.9 ) node {$\bullet$};
\draw (-0.3, 0.8) node {\tiny $1$};
\draw (0, -0.25) node {\scriptsize $\Sigma_3$};
\draw [-,very thick] (0,0) -- (0,0.3);
\draw [-,very thick] (0,0.3) -- (-0.3,0.6);
\draw [-,very thick] (0,0.3) -- (0.3,0.6);
\draw [-,very thick] (0.3, 0.6) -- (0.1,0.9);
\draw [-,very thick] (0.3, 0.6) -- (0.3,0.9);
\draw [-,very thick] (0.3, 0.6) -- (0.5,0.9);
\draw (1.2, 0) node  {\small $\bullet$}; 
\draw (1.2, 0.3) node  {\small $\bullet$};
\draw (0.9 , 0.6) node  {\small $\bullet$};
\draw (1.5 , 0.6) node  {\small $\bullet$};
\draw (1.5, 0.9 ) node {\small $\bullet$};
\draw (1.3, 0.9 ) node {\small $\bullet$};
\draw (1.7, 0.9 ) node {\small $\bullet$};
\draw (1.3, 1.1) node {\tiny $1$};
\draw (1.2, -0.25) node {\scriptsize $\Sigma_2$};
\draw [-,semithick] (1.2,0) -- (1.2,0.3);
\draw [-,semithick] (1.2,0.3) -- (0.9,0.6);
\draw [-,semithick] (1.2,0.3) -- (1.5,0.6);
\draw [-,semithick] (1.5, 0.6) -- (1.3,0.9);
\draw [-,semithick] (1.5, 0.6) -- (1.5,0.9);
\draw [-,semithick] (1.5, 0.6) -- (1.7,0.9);
\draw (2.4, 0) node  {\small $\bullet$}; 
\draw (2.4, 0.3) node  {\small $\bullet$};
\draw (2.1 , 0.6) node  {\small $\bullet$};
\draw (2.3 , 0.6) node  {\small $\bullet$};
\draw (2.7 , 0.6) node  {\small $\bullet$};
\draw (2.6, 0.9 ) node {\small $\bullet$};
\draw (2.8, 0.9 ) node {\small $\bullet$};
\draw (2.1, 0.8) node {\tiny $1$};
\draw (2.4, -0.25) node {\scriptsize $\Sigma_2$};
\draw [-,semithick] (2.4,0) -- (2.4,0.3);
\draw [-,semithick] (2.4,0.3) -- (2.1,0.6);
\draw [-,semithick] (2.4,0.3) -- (2.3,0.6);
\draw [-,semithick] (2.4,0.3) -- (2.7,0.6);
\draw [-,semithick] (2.7, 0.6) -- (2.6,0.9);
\draw [-,semithick] (2.7, 0.6) -- (2.8,0.9);
\draw (3.6, 0) node  {\small $\bullet$}; 
\draw (3.6, 0.3) node  {\small $\bullet$};
\draw (3.3 , 0.6) node  {\small $\bullet$};
\draw (3.5 , 0.6) node  {\small $\bullet$};
\draw (3.9 , 0.6) node  {\small $\bullet$};
\draw (3.8, 0.9 ) node {\small $\bullet$};
\draw (4, 0.9 ) node {\small $\bullet$};
\draw (3.8, 1.1) node {\tiny $1$};
\draw (3.6, -0.25) node {\scriptsize $\Sigma_2$};
\draw [-,semithick] (3.6,0) -- (3.6,0.3);
\draw [-,semithick] (3.6,0.3) -- (3.3,0.6);
\draw [-,semithick] (3.6,0.3) -- (3.5,0.6);
\draw [-,semithick] (3.6,0.3) -- (3.9,0.6);
\draw [-,semithick] (3.9, 0.6) -- (3.8,0.9);
\draw [-,semithick] (3.9, 0.6) -- (4,0.9);
\end{tikzpicture} 
\end{subfigure} \\ \hline
$-3$ & \begin{subfigure}[h]{5cm}
   \begin{tikzpicture}[xscale=1,yscale=1]
\draw (0, 0) node  {\small $\bullet$}; 
\draw (0, 0.3) node  {\small $\bullet$};
\draw (-0.3 , 0.6) node  {\small $\bullet$};
\draw (0.3 , 0.6) node  {\small $\bullet$};
\draw (0, 0.9 ) node {\small $\bullet$};
\draw (0.5, 0.9 ) node {\small $\bullet$};
\draw (0.6, 1.2 ) node {\small $\bullet$};
\draw (0.4, 1.2 ) node {\small $\bullet$};
\draw (-0.3, 0.8) node {\tiny $1$};
\draw (0, -0.25) node {\scriptsize $\Sigma_2$};
\draw [-,semithick] (0,0) -- (0,0.3);
\draw [-,semithick] (0,0.3) -- (-0.3,0.6);
\draw [-,semithick] (0,0.3) -- (0.3,0.6);
\draw [-,semithick] (0.3, 0.6) -- (0,0.9);
\draw [-,semithick] (0.3, 0.6) -- (0.5,0.9);
\draw [-,semithick] (0.5, 0.9) -- (0.6,1.2);
\draw [-,semithick] (0.5, 0.9) -- (0.4,1.2);
\draw (1.2, 0) node  {\small $\bullet$}; 
\draw (1.2, 0.3) node  {\small $\bullet$};
\draw (0.9 , 0.6) node  {\small $\bullet$};
\draw (1.5 , 0.6) node  {\small $\bullet$};
\draw (1.2, 0.9 ) node {\small $\bullet$};
\draw (1.7, 0.9 ) node {\small $\bullet$};
\draw (1.8, 1.2 ) node {\small $\bullet$};
\draw (1.6, 1.2 ) node {\small $\bullet$};
\draw (1.2, 1.1) node {\tiny $1$};
\draw (1.2, -0.25) node {\scriptsize $\Sigma_2$};
\draw [-,semithick] (1.2,0) -- (1.2,0.3);
\draw [-,semithick] (1.2,0.3) -- (0.9,0.6);
\draw [-,semithick] (1.2,0.3) -- (1.5,0.6);
\draw [-,semithick] (1.5, 0.6) -- (1.2,0.9);
\draw [-,semithick] (1.5, 0.6) -- (1.7,0.9);
\draw [-,semithick] (1.7, 0.9) -- (1.8,1.2);
\draw [-,semithick] (1.7, 0.9) -- (1.6,1.2);
\draw (2.4, 0) node  {\small $\bullet$}; 
\draw (2.4, 0.3) node  {\small $\bullet$};
\draw (2.1 , 0.6) node  {\small $\bullet$};
\draw (2.7 , 0.6) node  {\small $\bullet$};
\draw (2.4, 0.9 ) node {\small $\bullet$};
\draw (2.9, 0.9 ) node {\small $\bullet$};
\draw (3, 1.2 ) node {\small $\bullet$};
\draw (2.8, 1.2 ) node {\small $\bullet$};
\draw (2.8, 1.4) node {\tiny $1$};
\draw (2.4, -0.25) node {\scriptsize $1$};
\draw [-,semithick] (2.4,0) -- (2.4,0.3);
\draw [-,semithick] (2.4,0.3) -- (2.1,0.6);
\draw [-,semithick] (2.4,0.3) -- (2.7,0.6);
\draw [-,semithick] (2.7, 0.6) -- (2.4,0.9);
\draw [-,semithick] (2.7, 0.6) -- (2.9,0.9);
\draw [-,semithick] (2.9, 0.9) -- (3,1.2);
\draw [-,semithick] (2.9, 0.9) -- (2.8,1.2);
\draw (3.8, 0) node  {\small $\bullet$}; 
\draw (3.8, 0.3) node  {\small $\bullet$};
\draw (3.5 , 0.6) node  {\small $\bullet$};
\draw (4.1 , 0.6) node  {\small $\bullet$};
\draw (3.4, 0.9 ) node {\small $\bullet$};
\draw (3.6, 0.9 ) node {\small $\bullet$};
\draw (4, 0.9 ) node {\small $\bullet$};
\draw (4.2, 0.9 ) node {\small $\bullet$};
\draw (3.4, 1.1) node {\tiny $1$};
\draw (3.8, -0.25) node {\scriptsize $\Sigma_2$};
\draw [-,semithick] (3.8,0) -- (3.8,0.3);
\draw [-,semithick] (3.8,0.3) -- (3.5,0.6);
\draw [-,semithick] (3.8,0.3) -- (4.1,0.6);
\draw [-,semithick] (4.1, 0.6) -- (4,0.9);
\draw [-,semithick] (4.1, 0.6) -- (4.2,0.9);
\draw [-,semithick] (3.5, 0.6) -- (3.6,0.9);
\draw [-,semithick] (3.5, 0.6) -- (3.4,0.9);
\end{tikzpicture} 
\end{subfigure}
    \end{tabular}
  \end{center}
\caption{The $(1\times \Sigma_3)$-equivariant arboreal cells of $\partial_4$ with symmetry groups. \label{fig:d4}}
\end{figure}

 There is 
exactly one $(-1)$-cell of $\partial_{1+p}$ corresponding to $T_{1+p}$, which has symmetry group $(1\times \Sigma_p)$, as for $-2$-cells, 
there are $T^{1,1+p}$ with symmetry group $1\times \Sigma_p$, and cells with symmetry groups isomorphic to $G$ where $G\subset 1\times\Sigma_p$, 
with $p\nmid |G|$. So 
the $d_1$ in this spectral sequence goes from $H_*(S^{-1}_{h1\times \Sigma_p})\to H_*(\Sigma S^{-2}_{h1\times \Sigma_p})\bigoplus \{\text{other cells}\}$. Its image is entirely contained in $H_*(\Sigma S^{-2}_{h1\times \Sigma_p})$, 
which is representing the cell $T^{1,1+p}$, by an argument similar to the one given in the proof of Lemma \ref{Lem:Syl}, since the other cells have trivial homology except in one degree, but the map is trivial onto that degree. So we need to study the map 
\begin{align*}
H_*(S^{-1}_{h1\times \Sigma_p})\to H_{*-1}(\Sigma S^{-2}_{h1\times \Sigma_p}).
\end{align*}
We know from Lemma \ref{d_1=Tr} that this map is induced by the transfer \linebreak$1\times \Sigma_p\subset 1\times \Sigma_p$, and is therefore an isomorphism. Therefore we can now conclude that the cell $T^{1,1+p}$ does not represent anything non-trivial in the homology of $\partial_{1+p}\wedge_{h1\times \Sigma_p}\wedge S^{i+pj}$, and hence we are done.
\end{proof}

\section{The Computation for the Spheres}
As our goal is a computation of the free algebra over $\partial_*$, the first case will be the algebras generated by a sphere. We will proceed by parity, in the sense that we will  compute $H_*\mathbb{D}_n(X)$ first in the case when $X$ is an odd sphere, and then in the case when $X$ is an even sphere.

Computing $H_*\mathbb{P}(X)$, for $X$ a sphere, will allow us to leverage standard homology tricks, in our main result, Theorem \ref{Main} below, to compute $H_*\mathbb{P}(X)$ for any spectrum, $X$.

\subsection{The Odd Dimensional Case}
 In \cite[Thm. 3.16]{AM99} there is the following result:
\begin{Theo}[Arone-Mahowald, \cite{AM99}] \label{AMbasis}
$H_*(\mathbb{D}_{p^k}(S^{2l+1}))$ is the free graded $\mathbb{F}_p$ vector space on generators
\begin{align*}
\big\{[\beta^{\epsilon_1}Q^{s_1}\wr\cdots \wr \beta^{\epsilon_k}Q^{s_k}\iota] \big|s_k\geq l,\  s_i>ps_{i+1}-\epsilon_{i+1}\, \forall i\big\},
\end{align*}
where as usual $s_i\in \mathbb{N}_0$ and $\epsilon_i=0,1$, and $\iota$ is a generator of $H_{2l+1}S^{2l+1}$. With $\big|[\beta^{\epsilon_1}Q^{s_1}\wr\cdots \wr \beta^{\epsilon_k}Q^{s_k}\iota]\big|=2l+1+2(p-1)(s_1+\ldots+s_k)-\epsilon_1-\ldots -\epsilon_k-k$.

If $i\neq p^k$ for any $k$ then $H_*\mathbb{D}_i(S^{2l+1})=0$.
\end{Theo}
We call integer sequences of the form $(\epsilon_1,s_1,\ldots,\epsilon_k,s_k;l)$ completely unadmissable if $s_k\geq l$, and $s_i>ps_{i+1}-\epsilon_{i+1}$ for all $i$.
Recall that this description was found since there is a surjective homomorphism 
\begin{align*}
\Sigma^{-k}H_*((S^{2l+1})^{\wedge p^k}_{h\Sigma_p^{\wr k}})&\to H_*(\mathbb{D}_{p^k}(S^{2l+1})) \\
\beta^{\epsilon_1}Q^{s_1}\wr \cdots \wr \beta^{\epsilon_k}Q^{s_k}\iota & \mapsto [\beta^{\epsilon_1}Q^{s_1}\wr\cdots \wr \beta^{\epsilon_k}Q^{s_k}\iota ]
\end{align*}
coming from the spectral sequence arising from the simplicial filtration of $\partial_*$, see for example \cite[pf of thm. 1.5.1]{Beh12}. They arise since in the $(-k)$-cell labelled by $T_{p,k}$ the element $\beta^{\epsilon_1}Q^{s_1}\wr \cdots \wr \beta^{\epsilon_k}Q^{s_k}\iota$ is always a cycle, and we therefore get a group theoretical map. We are going to name this map $T_{p,k}$, as it comes from the cell represented by the tree $T_{p,k}$ in the simplicial filtration.

 Note that for $X$ a spectrum, we have that $\bigvee_i \mathbb{D}_i(X)$ is an algebra 
over $\partial_*$, and hence we can think of the Lie power operations as being: 
\begin{align*}
\overline{\beta^\epsilon Q^i}:H_*(\mathbb{D}_{p^k}(X))\to H_{*+2(p-1)i-\epsilon-1}(\mathbb{D}_{p^{k+1}}(X)).
\end{align*}
 
\begin{Prop} \label{AMP=LieP}
We have in $H_*(\mathbb{D}_{p^k}(S^{2l+1}))$ that 
\begin{align*}
[\beta^{\epsilon_1}Q^{s_1}\wr\cdots \wr \beta^{\epsilon_k}Q^{s_k}\iota]=\overline{\beta^{\epsilon_1}Q^{s_1}}\cdots \overline{\beta^{\epsilon_k}Q^{s_k}}\iota ,
\end{align*} where $\iota\in H_{2l+1}S^{2l+1}$ is a generator, for any completely unadmissable sequences $(\epsilon_1,s_1,\ldots,\epsilon_k,s_k;l)$.
\end{Prop}
\begin{proof}
We prove this by induction by relating our Lie power operations to the computation in \cite{AM99}. For ease of notation let $X=S^{2l+1}$. For the induction base case let $i\geq l$ and $\epsilon=0,1$. Then the following diagram commutes
\xymat{H_{*+1}(X)\ar[r]^-{\overline{\beta^\epsilon Q^i}} & H_{*+2(p-1)i-\epsilon}(\mathbb{D}_p(X)) \\
\Sigma^{-1}H_*(X)\ar@{=}[u]\ar[r]^-{q_{i,\epsilon}} & \Sigma^{-1}H_{*+2(p-1)i-\epsilon}(X^{\wedge p}_{h\Sigma_p})\ar[u]^{T_p}}
by how we defined $\overline{\beta^\epsilon Q^i}$ by the cell $T_p$, and by the fact that $T_p$ in both of the cellular filtrations of $\partial_p$ is the only $(-1)$-cell. This shows that the element called $[\beta^\epsilon Q^i \iota]$ is equal to $\overline{\beta^\epsilon Q^i}(\iota)$ for $\iota$ a generator of $H_{2l+1}(X)$. 

By converting levelled trees to metric trees, then grafting, and converting back to levelled trees, we can see that under the map $\partial_p\wedge \partial_{p^k}^{\wedge p} \stackrel{\xi}{\to} \partial_{p^{k+1}}$ we have in simplicial cells that $T_p\otimes T_{p,k}^{\otimes p}\mapsto T_{p,k+1}$. Pick $I=(\epsilon,i,\epsilon_k,i_k,\ldots \epsilon_1,i_1,l)$ to be a completely unadmissable sequence  and study the following diagram:
\xymat{H_{*+1}\mathbb{D}_{p^k}(X) \ar[r]^-{\overline{\beta^\epsilon Q^i}} & H_*(\partial_p \wedge_{h\Sigma_p}\mathbb{D}_{p^k}(X)) \ar[r]^\xi & H_*\mathbb{D}_{p^{k+1}}(X)\\
&\Sigma^{-1}H_*(\Sigma^{-k}X^{\wedge p^k}_{h\Sigma_p^{\wr k}})^{\wedge p}_{h\Sigma_p} \ar[u]^{T_p\wedge (T_{p,k})^{\wedge p}} & \\
\Sigma^{-k-1}H_*X^{\wedge p^k}_{h\Sigma_p^{\wr k}} \ar@{->>}[uu]^{T_{p,k}} \ar[r]^-{q_{i,\epsilon}} & \Sigma^{-1} \Sigma^{-k}H_*(X^{\wedge p^k}_{h\Sigma_p^{\wr k}})^{\wedge p}_{h\Sigma_p} \ar[u]^{\tau} \ar[r] & \Sigma^{-k-1} H_*X^{\wedge p^{k+1}}_{h\Sigma_p^{\wr k+1}} \ar@{->>}[uu]_{T_{p,k+1}}}
where $\tau$ sends an element $\beta^{\epsilon_{k+1}}Q^{i_{k+1}}\wr \beta^{\epsilon_{k}}Q^{i_{k}} \wr \ldots \wr \beta^{\epsilon_{k+1}}Q^{i_{k+1}}$ to the element $\beta^{\epsilon_{k+1}}Q^{i_{k+1}}\wr \sigma^{-k} \big(\beta^{\epsilon_{k}}Q^{i_{k}} \wr \ldots \wr \beta^{\epsilon_{k+1}}Q^{i_{k+1}}\big)$, and where $\sigma^{-k}$ is the $k$-fold desuspension.
By the argument above this diagram commutes. Further, if we start in the lower left corner with the element $\beta^{\epsilon_1}Q^{i_1}\wr\cdots \wr\beta^{\epsilon_k}Q^{i_k}\iota$, 
then chasing through the upper left corner gives us $\overline{\beta^\epsilon Q^i}\big(\overline{\beta^{\epsilon_1}Q^{i_1}}(\cdots \overline{\beta^{\epsilon_k}Q^{i_k}}(\iota)\ldots \big)$ in the upper right corner by the induction hypothesis, and chasing through the lower right corner gives us $[\beta^{\epsilon_1}Q^{i_1}\wr\cdots \wr\beta^{\epsilon_k}Q^{i_k}\iota]$ in the upper right corner.
 \end{proof}

\subsection{The Even Dimensional Case}
The direct computation in the odd dimensional case in \cite[Thm. 3.16]{AM99} does not extend to the case of even spheres. We will therefore rely on the following result of the EHP sequence in functor calculus. 

From \cite[Cor. 2.1.4]{Beh12} we know that the EHP sequence induces the following result:
\begin{Prop}[Behrens, \cite{Beh12}] \label{BehEHP}
For $n\geq 1$ the following are fiber sequences of spectra:
\begin{align}
\mathbb{D}_{2m}(S^n) \stackrel{E}{\to} & \Sigma^{-1} \mathbb{D}_{2m}(S^{n+1}) \stackrel{H}{\to}  \mathbb{D}_m(S^{2n+1}) \label{EHPev} \\
\mathbb{D}_{2m-1}(S^n) \stackrel{E}{\to}& \Sigma^{-1} \mathbb{D}_{2m-1}(S^{n+1}) \stackrel{H}{\to}  {*} \label{EHPod} 
\end{align}
\end{Prop}
Let $F:Top_*\to Top_*$ be a finitary homotopy functor, and let $\partial_*(F)$ denote its derivatives. 
Then as modules over $\partial_*$ we have from \cite[Ex. 19.4]{AC11} the following identification $\partial_m(\Omega\Sigma)\simeq \Sigma^{-1}\partial_m\wedge S^m$, so the suspension natural transformation $E:X\to \Omega \Sigma X$ induces a map $E:\partial_m \to \Sigma^{-1}\partial_m\wedge S^m$ compatible with the $\partial_*$-module structure. 
Therefore we get, for any spectrum $X$ and $m\in \mathbb{N}_0$, natural maps $E: \mathbb{D}_m(X)\to \Sigma^{-1}\mathbb{D}_m(\Sigma X)$ such that if $x\in H_*\mathbb{D}_m(X)$, then $E: \mathbb{D}_{pm}(X)\to \Sigma^{-1}\mathbb{D}_{pm}(\Sigma X)$ satisfies $E(\overline{\beta^\epsilon Q^i}(x))=\overline{\beta^\epsilon Q^i}(E(x))$. Similarly, the map $H$ preserves Lie power operations. 

For some time it was unknown how to extend these relations between the odd dimensional and even dimensional cases to the negative spheres. This was solved by Brantner in his thesis \cite[Section 4.1.3, Free Lie algebras on nonconnective spectra]{Bra17} from which we have the following result:
\begin{Lem}[Brantner, \cite{Bra17}] \label{Bran}
For any $n$ we get long exact sequences:
\begin{align*}
\ldots \to H_*\mathbb{D}_{2m}(S^n) \stackrel{E}{\to} & H_*\Sigma^{-1} \mathbb{D}_{2m}(S^{n+1}) \stackrel{H}{\to}  H_*\mathbb{D}_m(S^{2n+1}) \to \ldots
\end{align*}
and isomorphisms $E:H_*\mathbb{D}_{2m-1}(S^n)\stackrel{\cong}{\to} H_*\Sigma^{-1} \mathbb{D}_{2m-1}(S^{n+1})$.
\end{Lem}
In fact this holds even if we replace $H\mathbb{F}_p$ with any complex oriented cohomology theory. The proof comes from the fact that if $E$ is complex oriented, and $V$ is any complex $G$-representation then $E\wedge S^V$ and $E\wedge S^{\dim V}$ are equivalent as naive $G$-spectra. So in particular if $n\geq 0$ then $\Sigma^{2mn}H\mathbb{F}_p \wedge \partial_m \wedge (S^{-n})^{\wedge m}$, where $\Sigma_m$ acts trivially on $\Sigma^{mn}$, and by permuting the coordinates $(S^{-n})^{\wedge m}$, is equivalent as a naive $\Sigma_m$-spectrum to $H\mathbb{F}_p \wedge \partial_ m \wedge (S^{n})^{\wedge m}$, which gives the result.   

\begin{Lem} \label{[i,i]}
$H_*\mathbb{D}_{2}(S^{2l})$ is generated by the element $[\iota,\iota]$, where $\iota\in H_*\mathbb{D}_1(S^{2l})\cong H_*(S^{2l})$ is a generator. 
\end{Lem}
\begin{proof}
Note that $\partial_2$ consists of a single $\Sigma_2$-fixed $(-1)$-cell labelled by $T_2$, the same cell that carries the bracket operation. A routine calculation shows that when $l$ is positive $\mathbb{D}_{2}(S^{21})\simeq \Sigma^{2l-1}(\mathbb{R}P^\infty / \mathbb{R}P^{2l-1})$, which has homology concentrated in dimension $4l-1$. Furthermore it is easy to see that the map 
\begin{align*}
\partial_2\wedge S^{2l} \wedge S^{2l}=\partial_2\wedge \mathbb{D}_1(S^{2l}) \wedge \mathbb{D}_1(S^{2l})\to \mathbb{D}_2(S^{2l})
\end{align*}
is a homology isomorphism taking the generator $T_2\otimes \iota \otimes \iota$ to $[\iota , \iota ]$. 

Using the argument in lemma \ref{Bran} we can easily extend this to negative $l$'s as well.
\end{proof}

We are now ready to give the full case for the even dimensional sphere:
\begin{Cor}
As a $\mathbb{F}_p$ vector space, $H_*\mathbb{D}_m(S^{2l})$ has a basis
\begin{itemize}
\item $\big\{\overline{\beta^{\epsilon_1}Q^{s_1}}\cdots \overline{\beta^{\epsilon_k}Q^{s_k}}\iota \big|s_k\geq l,\  s_i>ps_{i+1}-\epsilon_{i+1}\, \forall i\big\}$ when $m=p^k$ for some $k$,
\item $\big\{\overline{\beta^{\epsilon_1}Q^{s_1}}\cdots \overline{ \beta^{\epsilon_k}Q^{s_k}}[\iota,\iota] \big|s_k\geq 2l,\  s_i>ps_{i+1}-\epsilon_{i+1}\, \forall i\big\}$ when $m=2p^k$ for some $k$, and
\item $\emptyset$ when $m\neq 2p^k$, or $m\neq p^k$ for any $k$.
\end{itemize}
where $\iota\in H_*\mathbb{D}_1(S^{2l})=H_*(S^{2l})$ is a generator.
\end{Cor} 
\begin{proof}
This clearly follows from Proposition \ref{BehEHP}, Lemma \ref{Bran} and \ref{[i,i]}, Theorem \ref{AMbasis} and \ref{AMP=LieP}. 
\end{proof}

Note that for $p\neq 3$, both the Jacobi identity and Proposition \ref{BehEHP} give us that $[\iota_{2l},[\iota_{2l},\iota_{2l}]]$ is trivial in $H_*(\mathbb{D}_3(S^{2l}))$. For $p=3$ however we need to check this by hand (note this is similar to the fact that in characteristic 3, there are two different notions of Lie algebras dependent on whether or not to include the axiom $[x,[x,x]]=0$).  
\begin{Cor} \label{[x,[x,x]]}
Let $p=3$. In $H_*(\mathbb{D}_3(S^{2l}))$ the element $[\iota_{2l},[\iota_{2l},\iota_{2l}]]$ is trivial.
\end{Cor}
\begin{proof}
In dimension $6l-3$, $H_*\mathbb{D}_3(S^{2l-1})$ has one generator $\overline{\beta Q^l} \iota_{2l-1}$. Under the isomorphism $H_*\mathbb{D}_3(S^{2l-1})\simeq H_*\Sigma^{-1}\mathbb{D}_3(S^{2l})$ this element maps to $\sigma^{-1}\overline{\beta Q^l} \iota_{2l}$ under 
the identification, where $\sigma^{-1}$ is the desuspension. We therefore just need to show that $[\iota_{2l},[\iota_{2l},\iota_{2l}]]$ 
is not a non-trivial multiple of $\overline{\beta Q^l} \iota_{2l}$. Tracing through the definitions we see that $\overline{\beta Q^l} \iota_{2l}$ 
is carried by the cell labelled by the tree $T_3$, and $[\iota_{2l},[\iota_{2l},\iota_{2l}]]$ is carried by the cell labelled by the $\Sigma_3$ orbit of $T^{1,3}$. The spectral sequence of Lemma \ref{SS} computing $H_*\mathbb{D}_3(S^{2l})$ has only two lines, so 
therefore we get a long exact sequence
\begin{align*}
\ldots \to H_{6l-2}(\Sigma^{-2}S^{2l})^{\wedge 3}_{h\Sigma_{T^{1,3}}} \to H_{6l-2}\mathbb{D}_3(S^{2l}) \to H_{6l-2}(\Sigma^{-1}S^{2l})^{\wedge 3}_{h\Sigma_{T_3}} \stackrel{d_1}{\to} \ldots
\end{align*}
Clearly 
\xymat{H_{6l-2}\mathbb{D}_3(S^{2l}) \ar[r] & H_{6l-2}(\Sigma^{-1}S^{2l})^{\wedge 3}_{h\Sigma_{T_3}} \\
\overline{\beta Q^l}\iota_{2l} \ar@{|->}[r] & \sigma^{-1}\beta Q^l\iota_{2l}}
by definition, where $\sigma^{-1}$ denotes the desuspension, and $\Sigma_{T_3}\simeq \Sigma_3$. Hence the element $\sigma^{-2} \iota_{2l} \star \iota_{2l} \star \iota_{2l}\in H_{6l-2}(\Sigma^{-2}S^{2l})^{\wedge 3}_{h\Sigma_{T^{1,3}}}$, which maps to $[\iota_{2l},[\iota_{2l},\iota_{2l}]]$ in $H_{6l-2}\mathbb{D}_3(S^{2l})$, must map to the trivial element, since $\Sigma_{T^{1,3}}\simeq \Sigma_2$, and $\star$ denotes the Pontryagin product. This concludes the proof. 
\end{proof}

\section{Main Result}
Before stating our main result, we will need a bit of terminology:
\begin{Def}
If $M_*$ is a graded $\mathbb{F}_p$ vector space with basis $B$, then define $\mathcal{A}(M_*)$ to be the graded vector space with basis 
\begin{align*}
B \cup \bigcup_{k\geq 2}\left\{[a_1,[a_2,[\ldots, [a_{k-1},a_k]\ldots ]]\big|a_i\in B \right\}
\end{align*}
where $|[x,y]|=|x|+|y|-1$. Define the free shifted Lie algebra $\mathbf{s}\mathcal{L}(M_*)$ to be $\mathcal{A}_*$ modulo the relations for all $x,y,z\in \mathcal{A}(M_*)$
\begin{itemize}
\item $[x,y]+[x,z]=[x,y+z]$ (Linearity)
\item $[x,y]=(-1)^{|x||y|}[y,x]$ (Graded Commutativity)
\item $(-1)^{|x||z|}[x,[y,z]] + (-1)^{|y||z|}[y,[z,x]]+(-1)^{|z||y|}[z,[x,y]]=0$ (The Graded Jacobi Identity)
\item For $p=3$ $[x,[x,x]]=0$
\end{itemize}

We define the free shifted Lie algebra with Power operations, $\mathbf{s}\mathcal{L}_{\mathcal{P}}(M_*)$ to be
\begin{align*}
\bigcup_{k\geq 0}\left\{\overline{\beta^{\epsilon_1}Q^{s_1}}\cdots \overline{ \beta^{\epsilon_k}Q^{s_k}}x \big|x\in\mathbf{s}\mathcal{L}(M_*), \  s_k\geq \frac{|x|}{2},\  s_i>ps_{i+1}-\epsilon_{i+1} \forall i \right\}
\end{align*}
subject to linearity of the power operations.
\end{Def}
Clearly given any spectrum $X$ we have a map $\mathbf{s}\mathcal{L}_{\mathcal{P}}(H_*(X))\to H_*\mathbb{P}(X)$.

We are now ready to state our main theorem:
\begin{Theo}\label{Main}
If $X$ is a spectra, then the map $\mathbf{s}\mathcal{L}_{\mathcal{P}}(H_*(X))\to H_*\mathbb{P}( X)$ is an isomorphism of $\mathbb{F}_p$ vector spaces.
\end{Theo}
The proof follows from the results above and arguments completely analogous to the proof of Theorem 7.1 in \cite{Cam15}, which I will summarize here.
\begin{proof}
Step one is showing that the homology of $\mathbb{P}(X)$ depends only on the homology of $X$. In fact Antolin-Camarena shows that this is true for any operad in spectra. From here we need to prove the theorem in increasing generality. The first case is taken care of by noting that we have already proved the theorem in the case where $X$ is a sphere. From here we can use a result of Arone and Kankaanrinta \cite[Thm. 0.1]{ArKa98} to extend the result to when $X$ is a finite wedge of spheres, and then we note that this also gives us the case for arbitrary wedges of spheres, writing it as a filtered colimit of finite wedges of spheres, and recall that homology and free constructions commute with filtered colimits. The last case is for a general $X$. We simply pick a basis for $H_*(X)$, $\{x_i\}$, and then use the following homology isomorphism:
\begin{align*}
H_*(\bigvee_i S^{|x_i|}) \to H_*(X)
\end{align*} 
given by the sum of the $x_i$'s and that gives the full result.
\end{proof}

One could have hoped for a description of the relations the power operations satisfy, as was done for the $p=2$ case in \cite[Thm. 1.5.1]{Beh12}. The argument there relies on a good understanding of the homology of the James-Hopf map (see \cite{Kuh83}), which due to combinatorics appears harder to obtain for odd primes.
\begin{Con} \label{MixAd}
The Lie power operations satisfy the mixed Adem relations, see \cite{CLM76} II.3 for a statement of these.
\end{Con}
\begin{Rem} \label{RemMixAd}
Note that to prove Conjecture \ref{MixAd} for positive homology classes, it will suffice to prove the following: 

The transfer $H_*(\Sigma_{p^2})\to H_*(\Sigma_p\wr \Sigma_p)$ is given by 
\begin{align*}
\beta^{\epsilon}Q^iQ^j\mapsto \beta^\epsilon Q^iQ^j+ \text{Mixed Adem Relations}.
\end{align*}

The conjecture then follows from the fact that the kernel of the surjection $\Sigma^{-2}H_*(S^{2l+1})^{p^2}_{h\Sigma_p \wr \Sigma_p} \to H_*\mathbb{D}_{p^2}(S^{2l+1})$ from \cite[Thm. 3.16]{AM99} is given by the image of the transfer $H_*(S^{2l+1})^{p^2}_{h\Sigma_{p^2}} \to H_*(S^{2l+1})^{p^2}_{h\Sigma_p \wr \Sigma_p}$.
\end{Rem}
Note that Conjecture \ref{MixAd} and Remark \ref{RemMixAd} hold for $p=2$.
\bibliography{bibleo}

\newcommand{\etalchar}[1]{$^{#1}$}
\providecommand{\bysame}{\leavevmode\hbox to3em{\hrulefill}\thinspace}
\providecommand{\MR}{\relax\ifhmode\unskip\space\fi MR }
\providecommand{\MRhref}[2]{%
  \href{http://www.ams.org/mathscinet-getitem?mr=#1}{#2}
}
\providecommand{\href}[2]{#2}
\begin{thebibliography}{{Ant}16}

\bibitem[AC11]{AC11}
Gregory Arone and Michael Ching, \emph{Operads and chain rules for the calculus
  of functors}, Soci{\'e}t{\'e} math{\'e}matique de France, 2011.

\bibitem[AK98]{ArKa98}
Greg Arone and Marja Kankaanrinta, \emph{The homology of certain subgroups of
  the symmetric group with coefficients in lie (n)}, Journal of Pure and
  Applied Algebra \textbf{127} (1998), no.~1, 1--14.

\bibitem[AM99]{AM99}
Greg Arone and Mark Mahowald, \emph{The {G}oodwillie tower of the identity
  functor and the unstable periodic homotopy of spheres}, Inventiones
  mathematicae \textbf{135} (1999), no.~3, 743--788.

\bibitem[{Ant}15]{Cam15}
Omar {Antol{\'{\i}}n Camarena}, \emph{The mod 2 homology of free spectral lie
  algebras}, Ph.D. thesis, 2015.

\bibitem[{Ant}16]{Cam16}
\bysame, \emph{{The mod 2 homology of free spectral Lie algebras}}, ArXiv
  e-prints (2016).

\bibitem[Beh11]{Beh11}
Mark Behrens, \emph{The {G}oodwillie tower for ${S}^1$ and {K}uhn's theorem},
  Algebraic \& Geometric Topology \textbf{11} (2011), no.~4, 2453--2475.

\bibitem[Beh12]{Beh12}
\bysame, \emph{The {G}oodwillie tower and the {EHP} sequence}, vol. 218,
  American Mathematical Society, 2012.

\bibitem[{Bra}17]{Bra17}
Lukas {Brantner}, \emph{The {L}ubin-{T}ate theory of spectral lie algebras},
  Ph.D. thesis, 2017.

\bibitem[Chi05]{Chi05}
Michael Ching, \emph{Bar constructions for topological operads and the
  {G}oodwillie derivatives of the identity}, Geometry \& Topology \textbf{9}
  (2005), no.~2, 833--934.

\bibitem[CLM76]{CLM76}
Frederick~Ronald Cohen, Thomas~Joseph Lada, and Peter~J May, \emph{The homology
  of iterated loop spaces}.

\bibitem[GK{\etalchar{+}}94]{GiKa94}
Victor Ginzburg, Mikhail Kapranov, et~al., \emph{Koszul duality for operads},
  Duke mathematical journal \textbf{76} (1994), no.~1, 203--272.

\bibitem[Goo03]{Goo03}
Thomas~G Goodwillie, \emph{{C}alculus {III}: {T}aylor series}, Geometry \&
  Topology \textbf{7} (2003), 645--711.

\bibitem[HSS00]{HSS00}
Mark Hovey, Brooke Shipley, and Jeff Smith, \emph{Symmetric spectra}, Journal
  of the American Mathematical Society \textbf{13} (2000), no.~1, 149--208.

\bibitem[KM95]{KM95}
Igor Kriz and J~Peter May, \emph{Operads, algebras, modules and motives},
  Soci{\'e}t{\'e} math{\'e}matique de France, 1995.

\bibitem[Kuh83]{Kuh83}
Nicholas~J Kuhn, \emph{The homology of the {J}ames-{H}opf maps}, Illinois
  Journal of Mathematics \textbf{27} (1983), no.~2, 315--333.

\bibitem[Wei95]{Wei95}
Charles~A. Weibel, \emph{An introduction to homological algebra}, no.~38,
  Cambridge university press, 1995.

\end{thebibliography}
\bibliographystyle{amsalpha}
\end{document}